\documentclass[reqno]{amsart}

\usepackage{style}

\author{Mireille Boutin}
  \address{Department of Mathematics and Computer Science, Technische Universiteit Eindhoven, P.O.~Box 513, 5600 MB Eindhoven, Netherlands and Department of Mathematics, Purdue University, 150 N.~University St., West Lafayette IN 47907} 
\email{m.boutin@tue.nl}
\author{Gregor Kemper} \address{Technische Universit\"at M\"unchen,
  Department of Mathematics, Boltzmannstr. 3, 85748 Garching, Germany}
\email{kemper@tum.de}

\title[Path Tracking Using Echoes]{Path Tracking using Echoes in an Unknown Environment: the Issue of Symmetries and How to Break Them}
\date{\today}

\subjclass[2020]{51K99, 51-08, 70E60}
\keywords{localization, SLAM, Cayley-Menger matrix, path reconstruction, acoustic SLAM, Euclidean symmetries, hyperplane arrangements}

\begin{document}

\begin{abstract}
  This paper deals with the problem of reconstructing the path of a vehicle in an unknown environment consisting of planar structures using sound. 
    Many systems in the literature do this by using a loudspeaker and microphones mounted on a vehicle. 
  Symmetries in the environment lead to solution ambiguities for such systems.
  We propose to resolve this issue by placing the loudspeaker at a fixed location in the environment rather than on the vehicle. The question of whether this will remove ambiguities regardless of the environment geometry leads to a question about breaking symmetries that can be phrased in purely mathematical terms.
  We solve this question in the affirmative if the geometry is in dimension three or bigger, and give counterexamples in dimension two. Excluding the rare situations where the counterexamples arise, we also give an affirmative answer in dimension two. Our results lead to a simple path reconstruction algorithm for a vehicle carrying four microphones navigating within an environment in which a loudspeaker at a fixed position emits short bursts of sounds. This algorithm could be combined with other methods from the literature to construct a path tracking system for vehicles navigating within a potentially symmetric environment. 
\end{abstract}

\maketitle

\section*{Introduction} \label{sIntro}
Several systems have been proposed to use sound to track the path of a vehicle in an unknown environment (e.g. [\citenumber{krekovic2016echoslam},\citenumber{dumbgen2022blind},\citenumber{saqib2020model}]). In order to track the vehicle, the geometry of the environment must be at least partly reconstructed as the vehicle navigates within it. Thus we are talking about the problem of Simultaneous Localization and Mapping (SLAM), in which the path of a user is determined while the shape and position of obstacles and other physical structures in the environment is reconstructed. See for example the book by \mycite{durrant2006simultaneous} for a general introduction to the SLAM problem. Our focus in this paper is on doing so using sound. More specifically, we are interested in acoustic SLAM (aSLAM) where an omnidirectional loudspeaker is used to produce a short burst of sound and microphones capture the echoes of this sound as it bounces on the objects in the environment. Our main interest is the correct reconstruction of the path of the vehicle, not a precise reconstruction of all the details of the environment. However, our results could provide the basis for a system to perform the latter task as well.

We consider the problem of path tracking inside an environment in $\RR^n$. The case of $n=3$ is of most interest. For example a vehicle rolling on the ground of a house would hear the echoes reflected by floors, ceilings and walls in $\RR^3$. So even though the path of the vehicle may be restricted to a 2D floor in this case, the overall problem involves the (partial) reconstruction of a 3D environment. The 2D case is also of considerable interest. For instance, some obstacle detection systems reduce the problem to two dimensions (e.g., the Crazyflie drone and the e-puck robot in [\citenumber{dumbgen2022blind}]).

When there are echoes from more than one surface, it is a priori unclear what echo comes from which surface. 
The task of assigning a surface to an echo is called ``echo sorting" and is a current problem of interest (see for example [\citenumber{macwilliam2023simultaneous}]). In a 3D environment consisting of planar surfaces, the echoes corresponding to different surfaces can be sorted when each surface is ``heard" by at least four microphones in a known geometric configuration [\citenumber{Boutin:Kemper:2019}]. One popular algorithm  [\citenumber{DPWLV1}] uses five microphones. In this paper we are using four microphones.

Systems that use fewer than four microphones 
but some additional sensors have also been developed. 
For example, a smartphone equipped with a Visual-Inertial Odometry (VIO) unit is used in [\citenumber{shih2019can}]. Another example is BatMapper [\citenumber{zhou2017batmapper}], which combines the cell phone audio sensing capability with a gyroscope and accelerometer. 
But systems based on a vehicle carrying a speaker and a single microphone, and no other sensors, have been shown to lead to ambiguities in the reconstruction [\citenumber{krekovic2016look}, \citenumber{krekovic2020shapes}].

In this paper we highlight the fact that, with any number of microphones, path ambiguities are unavoidable if the loudspeaker is carried on the vehicle.  
These ambiguities stem from symmetries in the environment. For example, a vehicle situated in a rectangular room would be unable to determine in which corner of the room it is situated based on the geometry of the surfaces around it. This is illustrated in \cref{fRectangle}.

\begin{figure}[htbp]
      \centering
      \begin{tikzpicture}[scale=0.2]
        \draw[thick] (-8,-5)--(8,-5);
        \draw[thick] (-8,5)--(8,5);
        \draw[thick] (-8,-5)--(-8,5);
        \draw[thick] (8,-5)--(8,5);
        
        \draw[ultra thick,gray] (-5,-2)--(-4,-2);
        \draw[ultra thick,gray] (-5,-2)--(-6,-3);
        \draw[ultra thick,gray] (-5,-2)--(-6,-1);
        \draw[ultra thick,gray] (-4,-2)--(-3,-3);
        \draw[ultra thick,gray] (-4,-2)--(-3,-1);
        \filldraw[red] (-4.5,-2) circle (0.2);
        \filldraw[red] (-6,-3) circle (0.2);
        \filldraw[red] (-3.5,-2.5) circle (0.2);
        \filldraw[red] (-6,-3) circle (0.2);
        \filldraw[red] (-17/3,-4/3) circle (0.2);
        \filldraw[blue] (-3,-1) circle (0.2);
        
        \draw[ultra thick,gray] (5,2)--(4,2);
        \draw[ultra thick,gray] (5,2)--(6,3);
        \draw[ultra thick,gray] (5,2)--(6,1);
        \draw[ultra thick,gray] (4,2)--(3,3);
        \draw[ultra thick,gray] (4,2)--(3,1);
        \filldraw[red] (4.5,2) circle (0.2);
        \filldraw[red] (6,3) circle (0.2);
        \filldraw[red] (3.5,2.5) circle (0.2);
        \filldraw[red] (6,3) circle (0.2);
        \filldraw[red] (17/3,4/3) circle (0.2);
        \filldraw[blue] (3,1) circle (0.2);
      \end{tikzpicture}%
  \caption{{\bf Loudspeaker on the vehicle.} The microphones (red) and the loudspeaker (blue) are positioned on the vehicle. In both vehicle positions indicated, the echoes from a sound heard by the microphones will be exactly the same. So the positions are indistinguishable.}
  \label{fRectangle}
\end{figure}
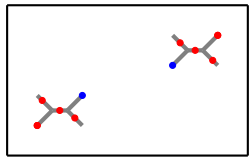

Symmetries in the environment make it mathematically impossible to determine the position of a vehicle carrying its own loudspeaker.
In order to address the issue, we put the loudspeaker in a fixed position in the environment rather than on the vehicle. 
This is illustrated in \cref{fRectangle2}.

\begin{figure}[htbp]
      \centering
      \begin{tikzpicture}[scale=0.2]
        \draw[thick] (-8,-5)--(8,-5);
        \draw[thick] (-8,5)--(8,5);
        \draw[thick] (-8,-5)--(-8,5);
        \draw[thick] (8,-5)--(8,5);

        \draw[dotted] (-2.75,5)--(-2,8);
        \draw[dashed] (-2.75,5)--(-4.5,-2);
        \draw[dashed] (-2.75,5)--(-2,2);

        \draw[dotted] (1.25,5)--(-2,8);
        \draw[dashed] (1.25,5)--(-2,2);
        \draw[dashed] (1.25,5)--(4.5,2);
        
        \filldraw[blue] (-2,2) circle (0.2);
        \filldraw[violet] (-14,2) circle (0.2);
        \filldraw[violet] (18,2) circle (0.2);
        \filldraw[violet] (-2,8) circle (0.2);
        \filldraw[violet] (-2,-12) circle (0.2);
        \filldraw[white] (-18,2) circle (0.2);
        
        \draw[very thick,gray] (-5,-2)--(-4,-2);
        \draw[very thick,gray] (-5,-2)--(-6,-3);
        \draw[very thick,gray] (-5,-2)--(-6,-1);
        \draw[very thick,gray] (-4,-2)--(-3,-3);
        \draw[very thick,gray] (-4,-2)--(-3,-1);
        \filldraw[red] (-4.5,-2) circle (0.2);
        \filldraw[red] (-6,-3) circle (0.2);
        \filldraw[red] (-3.5,-2.5) circle (0.2);
        \filldraw[red] (-6,-3) circle (0.2);
        \filldraw[red] (-17/3,-4/3) circle (0.2);
             
        \draw[very thick,gray] (5,2)--(4,2);
        \draw[very thick,gray] (5,2)--(6,3);
        \draw[very thick,gray] (5,2)--(6,1);
        \draw[very thick,gray] (4,2)--(3,3);
        \draw[very thick,gray] (4,2)--(3,1);
        \filldraw[red] (4.5,2) circle (0.2);
        \filldraw[red] (6,3) circle (0.2);
        \filldraw[red] (3.5,2.5) circle (0.2);
        \filldraw[red] (6,3) circle (0.2);
        \filldraw[red] (17/3,4/3) circle (0.2);
      \end{tikzpicture}%
  \caption{{\bf Loudspeaker at a fixed position.} The sound travels along the dashed lines. Virtually, it comes from the mirror points (violet). The fixed position of the loudspeaker (blue) is such that there is no symmetry among the mirror points. Vehicle positions are distinguishable.}
  \label{fRectangle2}
\end{figure}
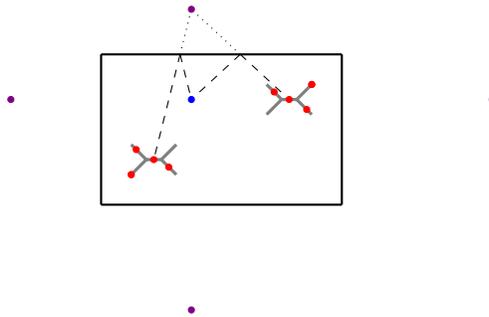

A main result of this paper (\cref{tSymmetries3D,tSymmetries2D}) is that putting the loudspeaker in a generic position in the environment makes the issue of symmetries disappear. This is what we mean by ``breaking symmetries.'' The theorems are phrased in purely mathematical terms. 

Our setup is as follows. A vehicle equipped with four microphones is moving inside an environment (e.g., a room or building) consisting of planar surfaces called ``walls." The vehicle may be flying in the room or moving on the ground. The wall positions are unknown and the microphone geometry is non-planar and known. More specifically, we know precisely the distance between the microphones, but the position of the microphone arrangement in the environment is unknown.
An omnidirectional loudspeaker is placed at a fixed unknown location inside the environment. The loudspeaker emits a short high-frequency signal (at a known time) and the signal bounces off the wall, creating echoes. The microphones listen to the sound of the original signal and the first order echoes to determine the distance to their respective source; the higher order echoes are discarded. We are assuming that the vehicle did not move while it was receiving the echoes, or that the change in position was so small that it can be neglected. We are also assuming that the vehicle was in a generic position at the time of reception. We are interested in reconstructing the path of the vehicle in the environment.

Our work builds on the methods from two previous
papers~[\citenumber{Boutin:Kemper:2019},\citenumber{Boutin:Kemper:2022}]. In [\citenumber{Boutin:Kemper:2019}] the vehicle is a drone with 3D freedom of translation and 3D freedom of rotation. In  [\citenumber{Boutin:Kemper:2022}], the vehicle is restricted to either 3D translation and yaw rotation (e.g., a hovering drone) or to movements on a ground plane (e.g., a car). 
In all cases, the  vehicle knows its own position and is equipped with four microphones: the goal is to determine the positions of the walls. In other words, we assume that the localization problem is solved, and so only the mapping problem remains. 
One problem is  that ``ghost walls,'' i.e., walls
that are detected but do not really exist, may appear. The main result
of~[\citenumber{Boutin:Kemper:2019}] is that ghost walls are only possible for few vehicle positions given full freedom of motion in 3D (translation and rotation) regardless of the wall positions. When the vehicle motion restricted, we showed in [\citenumber{Boutin:Kemper:2022}] that a few wall positions will lead to ghost walls being detected.

Having solved the problem of reconstructing the wall positions when the position of the vehicle is known, we now turn to the problem of determining the path of the vehicle in this paper. The methods of [\citenumber{Boutin:Kemper:2019},\citenumber{Boutin:Kemper:2022}] can be used to try to recover the geometry of the environment
in an arbitrary coordinate system (e.g., a local coordinate system for the microphone positions on the drone.) This can be repeated at various locations along the path of the vehicle. Since the coordinate system in which the wall positions are expressed will vary from one location to the next, it will be necessary to transform them to a common coordinate system; finding the transformations to the common coordinate system is equivalent to determining the path of the vehicle in that system. Determining this change of coordinate is the crux of the problem.

\cref{aLocate} lays out a procedure to do this. The algorithm is stated assuming that the environment has a 3D geometry, but it could be easily adapted to 2D. It uses ``mirror points,'' which are reflections of the loudspeaker position with respect to walls, as shown in \cref{fRectangle2}. Some of these mirror points, possibly together with the loudspeaker position, are located by the algorithm every time that the loudspeaker emits a sound burst. Together, we think of the mirror points and the loudspeaker position as ``sound sources.'' The idea is to match the detected sound sources to sound sources that have been detected from previous sound bursts, while the vehicle was at different positions. If enough sound sources can be matched, the current position of the vehicle is computed. This also yields the current attitude of the vehicle in terms of its principal axes. As the vehicle moves along its path, the algorithm builds a collection of detected sound sources. From these it is possible to determine the actual geometry of the walls, but we did not include this step into the formulation of the algorithm.  We view \cref{aLocate} more as a proof-of-concept than a ready-to-use procedure, as the ideas behind it can be combined with other methods 
to improve applicability, accuracy and efficiency. 
For example, the technical issues arising from the task of determining arrival times of echoes 
have been left out and could be addressed with existing methods from the literature (e.g., [\citenumber{cao2020effective}]). Also, the matching procedures could be improved by utilizing more sophisticated SLAM techniques such as graph-based SLAM [\citenumber{grisetti2010tutorial}].

\section{Problem Setup}
\newcommand{\ve}[1]{\mathbf{#1}}%

Consider an environment consisting of finite, planar surfaces in $\RR^n$.
A vehicle equipped with four non-coplanar microphones is moving inside the environment.
An omnidirectional loudspeaker is placed at a position in $\RR^n$ which need not be known.
The speaker produces a series of short high-pitch sounds. The times of the sound emissions are known. We are assuming that the vehicle is not moving while it is receiving the echoes of the different walls. The task at hand is to determine the position and orientation of the vehicle at those moments  where it receives the sound. Thus, we reconstruct the movements of the vehicle at discrete moments along the path.   

Each sound impulse bounces on some of the surfaces and is heard by some microphones.  Using the ray acoustic model, any surface that reflects a sound impulse can be represented by a ``{\bf mirror point}" (see \cref{fRectangle2}). The mirror point of a wall is the reflection of the loudspeaker position  with respect to the plane defined by the wall. This point can be viewed as a virtual source of sound. The set of all {\bf sound sources} for a given emission thus consists of the loudspeaker combined with all the virtual sound sources. This is a set of points whose geometry plays an important role in the following.


\section{Breaking symmetries} \label{sSymmetries}

\newcommand{\Aff}{\operatorname{Aff}}%
\newcommand{\brank}{\operatorname{b-rank}}
\newcommand{\re}{\operatorname{ref}}%


As explained in the introduction, we wish to put our loudspeaker in a
position such that the mirror points do not display any symmetry, even
if the environment geometry does. An example is shown in \cref{fRectangle2}.
The aim in this section is to show that regardless of the environment geometry,
this can be achieved by a generic choice of the loudspeaker position.
This section is phrased in purely mathematical terms. In particular, the
``environment'' is now treated as a finite set of hyperplanes in $\RR^n$.

A hyperplane arrangement, given by a finite set of affine hyperplanes
in $\RR^n$, can have symmetries, i.e., nonidentity elements of the
Euclidean group permuting the hyperplanes. Symmetries may be
``broken'' by choosing a point in $\RR^n$ (the loudspeaker position 
in our application) and then considering all
reflections of that point in the hyperplanes (the ``mirror points''),
instead of considering the hyperplane arrangement itself. Can the point
be placed such that all symmetries are broken, and no additional 
symmetries arise between the reflected points? Looking at situations 
such as the one shown in \cref{fRectangle2}, one might expect so, 
but at least in dimension two, the answer is {\em not in general}, as
\cref{rSymmetries}\ref{rSymmetriesC} below shows. However, the
following result, \cref{tSymmetries3D}, gives a positive answer in the
case of dimension~$\ge 3$ (\cref{rSymmetries}\ref{rSymmetriesB} makes
this precise). \cref{tSymmetries2D} then deals with the
two-dimensional case, thus qualifying the observation that breaking
symmetries is in general not possible. The results say that a generic
choice of point breaks all symmetries. In both theorems, the ``no
symmetries'' statement is made in a strong way: the set of reflections
of the chosen point has no nonidentity isometry. Neither do there
exist isometries between subsets of reflection points, as long as
those subsets are geometrically ``large enough.''

In the following, $\re_H(\ve v)$ denotes the reflection of a point
$\ve v \in \RR^n$ in an affine hyperplane $H \subset \RR^n$. The 
point~$\ve v$ can be thought of as the loudspeaker position, and the
$\re_H(\ve v)$ as the mirror points. In our application, $\ve v$ together 
with the $\re_H(\ve v)$ forms the set of sound sources.

\begin{theorem} \label{tSymmetries3D}%
  Let $\mathcal H$ be a finite set of affine hyperplanes in
  $\RR^n$. Then there is a nonzero polynomial
  $f \in \RR[x_1 \upto x_n]$ such that for all $\ve v \in \RR^n$ with
  $f(\ve v) \ne 0$ and for
  $H_1 \upto H_m,H'_1 \upto H'_m \in \mathcal H$ such that the normal
  vectors of $H_1 \upto H_m$ span a vector space of dimension $\ge 3$,
  we have: if the reflections $\ve w_i := \re_{H_i}(\ve v)$,
  $\ve w'_i := \re_{H'_i}(\ve v)$ of~$\ve v$ satisfy
  \[
    \lVert\ve w_i - \ve w_j\rVert = \lVert\ve w'_i - \ve w'_j\rVert
    \quad (1 \le i < j \le m),
  \]
  then $H_i = H'_i$ and therefore $\ve w_i = \ve w'_i$ for
  all~$i$. Moreover, for $H_1,H_2,H_3 \in \mathcal H$ with $H_1 \ne
  H_2$ we have
  \begin{equation} \label{eqIneq3}%
    \bigl\lVert\re_{H_1}(\ve v) - \ve v\bigr\rVert \ne
    \bigl\lVert\re_{H_2}(\ve v) - \ve v\bigr\rVert \ne
    \bigl\lVert\re_{H_1}(\ve v) - \re_{H_3}(\ve v)\bigr\rVert.
  \end{equation}
  In other words, each distance between~$\ve v$ and a reflection point
  is unique among the distances between~$\ve v$ and reflection points
  and distances between two reflection points.
\end{theorem}

The theorem will be proved together with \cref{tSymmetries2D} below.

\begin{rem} \label{rSymmetries}%
  \begin{enumerate}[label=(\alph*)]
  \item \label{rSymmetriesB} A hyperplane arrangement $\mathcal H$ in
    which all the normal vectors of the hyperplanes are contained in a
    two-dimensional subspace is itself ``morally'' two-dimensional. So
    let us assume that our hyperplane arrangement has three
    hyperplanes with linearly independent normal vectors. Then
    choosing the $\ve w_i$ as all the reflections of~$\ve v$ means
    that the hypothesis of \cref{tSymmetries3D} is met. So if~$\phi$
    is a Euclidean transformation that permutes the~$\ve w_i$, we can
    apply the theorem to the $\ve w_i$ and
    $\ve w'_i := \phi(\ve w_i)$. This yields $\ve w_i = \ve w'_i$,
    implying that~$\phi$ restricts to the identity on the affine space
    generated by $\ve w_1 \upto \ve w_m$. This makes it precise
    that~$\ve w_1 \upto \ve w_m$ have no symmetries.
  \item \label{rSymmetriesC} \newcommand{\rot}{\operatorname{rot}}%
    \cref{tSymmetries3D} says nothing in the case of dimension
    $n = 2$, since in that case no hyperplanes
    $H_1 \upto H_m$ can possibly meet the dimension hypothesis on the
    normal vectors. In fact, the following construction shows that in
    dimension two, breaking symmetries by taking reflections of a
    point is impossible for some arrangements of hyperplanes (which in
    2D are just lines). Let $H_1 \upto H_m \subset \RR^2$ be lines
    through the coordinate origin and let $\rot_{2 \phi} \in \SO_2$ be
    a rotation about an angle $2 \phi$, with~$\phi$ not a multiple
    of~$\pi$. We have
    \begin{equation} \label{eqRotRef}%
      \rot_{2 \phi} \circ \re_{H_i} = \rot_\phi \circ \re_{H_i} \circ
      \re_{H_i} \circ \rot_\phi \circ \re_{H_i} = \rot_\phi \circ
      \re_{H_i} \circ \rot_\phi^{-1} = \re_{\rot_\phi(H_i)},
    \end{equation}
    so it we set $H'_i := \rot_\phi(H_i)$, and, with $\ve v \in \RR^2$
    arbitrary, $\ve w_i := \re_{H_i}(\ve v)$,
    $\ve w'_i := \re_{H'_i}(\ve v)$, then
    \[
      \lVert\ve w_i - \ve w_j\rVert = \lVert\rot_{2 \phi}(\ve w_i) -
      \rot_{2 \phi}(\ve w_j)\rVert \underset{\cref{eqRotRef}}{=}
      \lVert\ve w'_i - \ve w'_j\rVert,
    \]
    but $H_i \ne H'_i$ and $\ve w_i \ne \ve w'_i$. So the assertion of
    \cref{tSymmetries3D} fails.

    To turn this into an example about symmetries, choose
    $\phi = \pi/k$ with $k \ge 3$ an integer, choose a line $H$
    through the origin, and set
    $\mathcal H := \{H_i:= \rot_{i \phi}(H) \mid i = 0 \upto
    k-1\}$. Then the symmetry group of $\mathcal H$ is a dihedral group
    of order $4 k$, generated by $\rot_\phi$ and $\re_H$. For the
    reflections $\ve w_i = \re_{H_i}(\ve v)$, \cref{eqRotRef} yields
    $\ve w_{i+1} = \rot_{2 \phi}(\ve w_i)$, setting
    $\ve w_k := \ve w_0$. So the $\ve w_i$ form a regular polygon
    with~$k$ vertices and dihedral symmetry group of order $2
    k$. \cref{fTriangle} shows this for~$k = 3$.
    
    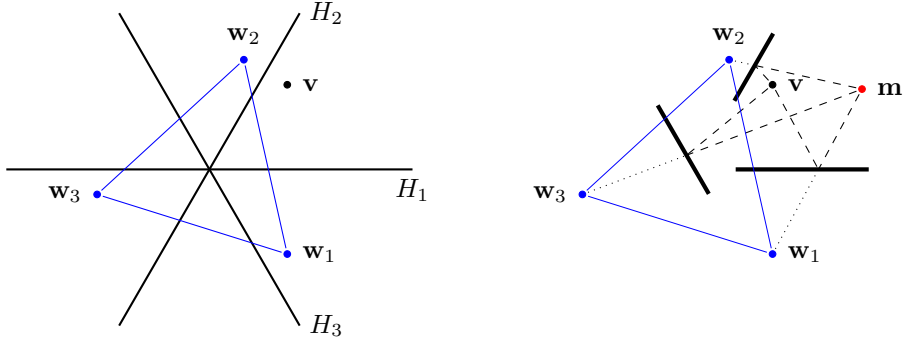
\begin{figure}[htbp]
      \centering
      \begin{tikzpicture}[scale=0.6]
        \draw[thick] (-4.5,0)--(4.5,0) node[below,at end]{$H_1$};%
        \draw[thick] (-2,-3.464)--(2,3.464) node[right,at
        end]{$H_2$};%
        \draw[thick] (-2,3.464)--(2,-3.464) node[right,at
        end]{$H_3$};%
        \node[circle,draw,outer sep=1pt, inner sep=0pt,minimum
        size=2.5pt,fill,label={right:$\ve v$}] (v) at (1.725,1.877) {};%
        \node[blue,circle,draw,outer sep=1pt, inner sep=0pt,minimum
        size=2.5pt,fill,label={right:$\ve w_1$}] (w1) at (1.725,-1.877) {};%
        \node[blue,circle,draw,outer sep=1pt, inner sep=0pt,minimum
        size=2.5pt,fill,label={above:$\ve w_2$}] (w2) at (0.763,2.432) {};%
        \node[blue,circle,draw,outer sep=1pt, inner sep=0pt,minimum
        size=2.5pt,fill,label={left:$\ve w_3$}] (w3) at (-2.488,-0.555) {};%
        \draw[blue] (w1) -- (w2);%
        \draw[blue] (w1) -- (w3);%
        \draw[blue] (w2) -- (w3);%
      \end{tikzpicture}%
      \qquad
      \begin{tikzpicture}[scale=0.6]
        \draw[white] (-4.5,0)--(4.5,0) node[below,at end]{$H_1$};%
        \draw[white] (-2,-3.464)--(2,3.464) node[right,at
        end]{$H_2$};%
        \draw[white] (-2,3.464)--(2,-3.464) node[right,at
        end]{$H_3$};%
        
        \draw[ultra thick] (0.908,0)--(3.856,0);%
        \draw[ultra thick] (0.879,1.522)--(1.739,3.012);%
        \draw[ultra thick] (-0.820,1.420)--(0.313,-0.542);%
        \node[circle,draw,outer sep=1pt, inner sep=0pt,minimum
        size=2.5pt,fill,label={right:$\ve v$}] (v) at (1.725,1.877) {};%
        \node[blue,circle,draw,outer sep=1pt, inner sep=0pt,minimum
        size=2.5pt,fill,label={right:$\ve w_1$}] (w1) at (1.725,-1.877) {};%
        \node[blue,circle,draw,outer sep=1pt, inner sep=0pt,minimum
        size=2.5pt,fill,label={above:$\ve w_2$}] (w2) at (0.763,2.432) {};%
        \node[blue,circle,draw,outer sep=1pt, inner sep=0pt,minimum
        size=2.5pt,fill,label={left:$\ve w_3$}] (w3) at (-2.488,-0.555) {};%
        \node[red,circle,draw,outer sep=1pt, inner sep=0pt,minimum
        size=2.5pt,fill,label={right:$\ve m$}] (m) at (3.707,1.781) {};%
        
        \draw[dashed] (v)--(-0.181,0.314);%
        \draw[dashed] (m)--(-0.181,0.314);%
        \draw[dotted] (w3)--(-0.181,0.314);%
        
        \draw[dashed] (v)--(2.742,0);%
        \draw[dashed] (m)--(2.742,0);%
        \draw[dotted] (w1)--(2.742,0);%
        
        \draw[dashed] (v)--(1.332,2.306);%
        \draw[dashed] (m)--(1.332,2.306);%
        \draw[dotted] (w2)--(1.332,2.306);%
        
        \draw[ultra thin,blue] (w1) -- (w2);%
        \draw[ultra thin,blue] (w1) -- (w3);%
        \draw[ultra thin,blue] (w2) -- (w3);%
      \end{tikzpicture}%
      \caption{{\bf Left:} no matter where~$\ve v$ is, its reflections~$\ve
        w_i$ in the $H_i$ always form an equilateral triangle. {\bf Right:} the same geometry, but now a microphone position (red) has been introduced and the hyperplanes have been reduced to limited walls. This shows that the geometry on the left can be realized in such a way that the echoes can actually be heard.}
      \label{fTriangle}
    \end{figure}
    
    Notice that the symmetry of the triangle that swaps~$\ve w_1$
    and~$\ve w_2$ is new since it is a reflection in a line that is
    not a symmetry axis of the hyperplane arrangement. \remend
  \end{enumerate}
  \renewcommand{\remend}{}
\end{rem}

The next result is a variant of \cref{tSymmetries3D} which does work
in dimension~$2$. It requires an additional hypothesis on the
hyperplane arrangement, which in dimension~$2$ just stipulates that
the intersection of three lines in the arrangement must be empty. So
one might say that precisely the situation from
\cref{rSymmetries}\ref{rSymmetriesC}, shown in \cref{fTriangle}, is
excluded.  The additional hypothesis is mild in the sense that a
generic hyperplane arrangement satisfies it. Comparing
\cref{tSymmetries3D,tSymmetries2D}, one sees that the extra hypothesis
allows to replace the ``dimension $\ge 3$'' in \cref{tSymmetries3D} by
``dimension $\ge 2$.'' Since
$\lVert\ve w_1 - \ve w_2\rVert = \lVert\ve w'_1 - \ve w'_2\rVert$
cannot be enough to conclude $\ve w_i = \ve w'_i$, it is also clear
that the hypothesis $\bigl|\{H_1 \upto H_m\}\bigr| \ge 3$ is needed in
\cref{tSymmetries2D}.

\begin{theorem} \label{tSymmetries2D}%
  Let $\mathcal H$ be a finite set of affine hyperplanes in $\RR^n$
  such that no three hyperplanes from $\mathcal H$ meet in
  codimension~$2$. Then there is a nonzero polynomial
  $f \in \RR[x_1 \upto x_n]$ such that for all $\ve v \in \RR^n$ with
  $f(\ve v) \ne 0$ and for
  $H_1 \upto H_m,H'_1 \upto H'_m \in \mathcal H$ such that the normal
  vectors of $H_1 \upto H_m$ span a vector space of dimension $\ge 2$
  and $\bigl|\{H_1 \upto H_m\}\bigr| \ge 3$, we have: if the
  reflections $\ve w_i := \re_{H_i}(\ve v)$,
  $\ve w'_i := \re_{H'_i}(\ve v)$ of~$\ve v$ satisfy
  \[
    \lVert\ve w_i - \ve w_j\rVert = \lVert\ve w'_i - \ve w'_j\rVert
    \quad (1 \le i < j \le m),
  \]
  then $H_i = H'_i$ and therefore $\ve w_i = \ve w'_i$ for
  all~$i$. Moreover, for $H_1,H_2,H_3 \in \mathcal H$ with $H_1 \ne
  H_2$ we have
  \begin{equation} \label{eqIneq2}%
    \bigl\lVert\re_{H_1}(\ve v) - \ve v\bigr\rVert \ne
    \bigl\lVert\re_{H_2}(\ve v) - \ve v\bigr\rVert \ne
    \bigl\lVert\re_{H_1}(\ve v) - \re_{H_3}(\ve v)\bigr\rVert.
  \end{equation}
  In other words, each distance between~$\ve v$ and a reflection point
  is unique among the distances between~$\ve v$ and reflection points
  and distances between two reflection points.
\end{theorem}

\begin{proof}[Proof of \cref{tSymmetries3D,tSymmetries2D}]
  We first need to construct the polynomial~$f$. It follows from
  \cref{lSymmetries}\ref{lSymmetriesA} (see below this proof) that for
  hyperplanes $H_1,H_2,H_3 \in \mathcal H$, the polynomial
  $g_{H_1,H_2,H'} \in \RR[x_1 \upto x_n]$ given by
  \[
    g_{H_1,H_2,H_3}(\ve v) = \bigl\lVert\re_{H_1}(\ve v) -
    \re_{H_3}(\ve v)\bigr\rVert^2 - \bigl\lVert\re_{H_2}(\ve v) - \ve
    v\bigr\rVert^2
  \]
  is nonzero. If $H_1 \ne H_2$, there is $\ve v \in \RR^n$ such that
  $\bigl\lVert\re_{H_1}(\ve v) - \ve v\bigr\rVert \ne
  \bigl\lVert\re_{H_2}(\ve v) - \ve v\bigr\rVert$: one needs to avoid
  the hyperplane consisting of all points that have equal distance to
  $H_1$ and $H_2$. So the polynomial $h_{H_1,H_2}$ defined by
  \[
    h_{H_1,H_2}(\ve v) = \bigl\lVert\re_{H_1}(\ve v) - \ve v\bigr\rVert^2 -
    \bigl\lVert\re_{H_2}(\ve v) - \ve v\bigr\rVert^2
  \]
  also is nonzero. The first part of the construction of~$f$ is now
  given by
  \[
    f_{\operatorname{part} 1} := \Biggl(\prod_{H_1,H_2,H_3 \in \mathcal
      H} g_{H_1,H_2,H_3}\Biggr) \Biggl(\prod_{\substack{H_1,H_2 \in \mathcal H \\
        H_1 \ne H_2}} h_{H_1,H_2}\Biggr).
  \]
  So if $f_{\operatorname{part} 1}(\ve v) \ne 0$
  then~\cref{eqIneq3} and~\cref{eqIneq2} are satisfied.
  
  The second part of the construction of~$f$ takes different routes
  for \cref{tSymmetries3D} or \ref{tSymmetries2D}. In the following,
  we will call affine hyperplanes \df{linearly independent} if their
  normal vectors are linearly independent. It is easy to see that~$k$
  hyperplanes are linearly independent if and only if they meet in
  codimension~$k$: just consider the system of linear equations for
  determining their intersection.
  \begin{description}
  \item[Case of \cref{tSymmetries3D}] We take three linearly
    independent hyperplanes $H_1,H_2,H_3 \in \mathcal H$ and three
    further hyperplanes $H'_1,H'_2,H'_3 \in \mathcal H$ such that
    \[
      \bigl\lVert\re_{H_i}(\ve v) - \re_{H_j}(\ve v)\bigr\rVert =
      \bigl\lVert\re_{H'_i}(\ve v) - \re_{H'_j}(\ve v)\bigr\rVert
      \quad (1 \le i < j \le 3) \quad \text{for all} \ \ve v \in
      \RR^n.
    \]
    By \cref{lSymmetries}\ref{lSymmetriesB}, this implies
    $H_i \cap H_j = H'_i \cap H'_j$ for all~$i$ and~$j$. Writing 
    $\Aff(M)$ for the affine subspace spanned by some point set 
    $M$ and using the linear independence of the $H_i$, we conclude
    \[
      H_1 = \Aff\bigl((H_1 \cap H_2) \cup (H_1 \cap H_3)\bigr) =
      \Aff\bigl((H'_1 \cap H'_2) \cup (H'_1 \cap H'_3)\bigr) \subseteq
      H'_1,
    \]
    so $H_1 = H'_1$. The same argument shows $H_i = H'_i$ for
    all~$i$. We write this as $(H_1,H_2,H_3) =
    (H'_1,H'_2,H'_3)$. Going to the contrapositive, we have shown that
    for $H_1,H_2,H_3 \in \mathcal H$ linearly independent and for
    $H'_1,H'_2,H'_3 \in \mathcal H$ with
    $(\underline H) := (H_1,H_2,H_3) \ne (H'_1,H'_2,H'_3) =:
    (\underline H')$, the polynomial
    $f_{\underline H,\underline H'} \in \RR[x_1 \upto x_n]$ defined by
    \[
      f_{\underline H,\underline H'}(\ve v) = \sum_{1 \le i < j \le 3}
      \Bigl(\bigl\lVert\re_{H_i}(\ve v) - \re_{H_j}(\ve
      v)\bigr\rVert^2 - \bigl\lVert\re_{H'_i}(\ve v) - \re_{H'_j}(\ve
      v)\bigr\rVert^2\Bigr)^2
    \]
    is nonzero. We set
    \[
      f_{\operatorname{part} 2} := \prod_{\substack{H_1,H_2,H_3 \in \mathcal H \\
          \text{linearly independent}}} \,\,\,
      \prod_{\substack{H'_1,H'_2,H'_3 \in \mathcal H \\ \text{with}
          (\underline H) \ne (\underline H')}} f_{\underline
        H,\underline H'}.
    \]
  \item[Case of \cref{tSymmetries2D}] Let $H_1,H_2 \in \mathcal H$ be
    linearly independent, which means that they are not parallel. Then
    \cref{lSymmetries}\ref{lSymmetriesC} tells us that for
    $H'_1,H'_2 \in \mathcal H$ with $\{H_1,H_2\} \ne \{H'_1,H'_2\}$
    the polynomial $f_{\underline H,\underline H'}$ given by
    \[
      f_{\underline H,\underline H'}(\ve v) = \bigl\lVert\re_{H_1}(\ve
      v) - \re_{H_2}(\ve v)\bigr\rVert^2 - \bigl\lVert\re_{H'_1}(\ve
      v) - \re_{H'_2}(\ve v)\bigr\rVert^2
    \]
    is nonzero. In this case we set 
    \[
      f_{\operatorname{part} 2} := \prod_{\substack{H_1,H_2 \in \mathcal H \\
          \text{linearly independent}}} \,\,\,
      \prod_{\substack{H'_1,H'_2 \in \mathcal H \ \text{with} \\
          \{H_1,H_2\} \ne \{H'_1,H'_2\}}} f_{\underline H,\underline
        H'}.
    \]
  \end{description}
  In both cases we set
  $f := f_{\operatorname{part} 1} \cdot f_{\operatorname{part} 2}$. So
  $f(\ve v) \ne 0$ implies~\cref{eqIneq3} and~\cref{eqIneq2}.

  To prove the other assertions of the theorems, let $\ve v \in \RR^n$
  with $f(\ve v) \ne 0$ and take
  $H_1 \upto H_m,H'_1 \upto H'_m \in \mathcal H$ as in the theorems
  such that the $\ve w_i = \re_{H_i}(\ve v)$ and
  $\ve w'_i := \re_{H'_i}(\ve v)$ satisfy
  \begin{equation} \label{eqWiWis}%
    \lVert\ve w_i - \ve w_j\rVert = \lVert\ve w'_i - \ve w'_j\rVert
    \quad (1 \le i < j \le m).
  \end{equation}
  Let $i \in \{1 \upto m\}$ be arbitrary. We need to show
  $H_i = H'_i$. Again the arguments for \cref{tSymmetries3D}
  and~\ref{tSymmetries2D} differ.

  \begin{description}
  \item[Case of \cref{tSymmetries3D}] By hypothesis we can choose~$j$
    and~$k$ in $\{1 \upto m\}$ such that $H_i,H_j,H_k$ have linearly
    independent normal vectors. Assume that $H_i \ne H'_i$. Then
    $f_{H_i,H_j,H_k,H'_i,H'_j,H'_k}(\ve v) \ne 0$ implies that at
    least one of the differences
    $\lVert\ve w_i - \ve w_j\rVert^2 - \lVert\ve w'_i - \ve
    w'_j\rVert^2$,
    $\lVert\ve w_i - \ve w_k\rVert^2 - \lVert\ve w'_i - \ve
    w'_k\rVert^2$, or
    $\lVert\ve w_j - \ve w_k\rVert^2 - \lVert\ve w'_j - \ve
    w'_k\rVert^2$ is nonzero, contradicting~\cref{eqWiWis}. We
    conclude $H_i = H'_i$, as desired.
  \item[Case of \cref{tSymmetries2D}] By hypothesis we can choose
    $j \in \{1 \upto m\}$ such that $H_i,H_j$ are linearly
    independent. It follows that $\{H_i,H_j\} = \{H'_i,H'_j\}$, since
    otherwise $f_{H_i,H_j,H'_i,H'_j}(\ve v) \ne 0$ would imply
    $\lVert\ve w_i - \ve w_j\rVert \ne \lVert\ve w'_i - \ve
    w'_j\rVert$, contradicting~\cref{eqWiWis}. Also by hypothesis
    there exists $k \in \{1 \upto m\}$ such that
    $H_i \ne H_k \ne H_j$. $H_k$ cannot be parallel to both $H_i$ and
    $H_j$. So $H_i$ and $H_k$ or $H_j$ and $H_k$ are linearly
    independent. In the first case, we get, as above,
    $\{H_i,H_k\} = \{H'_i,H'_k\}$, and in the second case we get
    $\{H_j,H_k\} = \{H'_j,H'_k\}$. But either case, together with
    $\{H_i,H_j\} = \{H'_i,H'_j\}$, implies $H_i = H'_i$, $H_j = H'_j$
    and $H_k = H'_k$.
  \end{description}
  In both cases we have seen that $H_i = H'_i$, which finishes the
  proof.
\end{proof}

The following lemma was used in the above proof.

\begin{lemma} \label{lSymmetries}%
  Assume the hypotheses of \cref{tSymmetries3D}.
  \begin{enumerate}[label=(\alph*)]
  \item \label{lSymmetriesA} Let $H_1,H_2,H_3 \in \mathcal H$. Then there
    exists $\ve v \in \RR^n$ such that
    $\bigl\lVert\re_{H_1}(\ve v) - \re_{H_3}(\ve v)\bigr\rVert \ne
    \bigl\lVert\re_{H_2}(\ve v) - \ve v\bigr\rVert$.
  \item \label{lSymmetriesB} Let $H_1,H_2,H'_1,H'_2 \in \mathcal H$. If
    $\bigl\lVert\re_{H_1}(\ve v) - \re_{H_2}(\ve v)\bigr\rVert =
    \bigl\lVert\re_{H'_1}(\ve v) - \re_{H'_2}(\ve v)\bigr\rVert$ for
    all $\ve v \in \RR^n$, then either
    $H_1 \cap H_2 = H'_1 \cap H'_2$, or $H_1 = H_2$ and $H'_1 = H'_2$.
  \item \label{lSymmetriesC} Under the hypotheses of \cref{tSymmetries2D}, the
    assertion of \cref{lSymmetriesB} can be sharpened as follows: either
    $\{H_1,H_2\} = \{H'_1,H'_2\}$, or $H_1$ is parallel to $H_2$ and
    so is $H'_1$ to $H'_2$.
  \end{enumerate}
\end{lemma}

\begin{proof}
  \begin{enumerate}
  \item[\ref{lSymmetriesB}] Take $\ve v \in H_1 \cap H_2$. Then
    $\re_{H_1}(\ve v) = \ve v = \re_{H_2}(\ve v)$, so
    $\re_{H'_1}(\ve v) = \re_{H'_2}(\ve v) =: \ve w$ by hypothesis. If
    $\ve w = \ve v$, this implies $\ve v \in H'_1 \cap H'_2$. On the
    other hand, if $\ve w \ne \ve v$, then $H'_1$ consists of all
    points that have equal distance to~$\ve v$ and to~$\ve w$, and the
    same for $H'_2$; hencce $H'_1 = H'_2$. We conclude that either
    $H_1 \cap H_2 \subseteq H'_1 \cap H'_2$ or $H'_1 = H'_2$. From
    this \cref{lSymmetriesB} follows by reversing the roles of the
    $H_i$ and the $H'_i$.
  \item[\ref{lSymmetriesC}] The assertion of \cref{lSymmetriesB}
    holds, but now we have the additional hypothesis of
    \cref{tSymmetries2D}. Assume that $H_1$ is not parallel to $H_2$,
    which implies $H_1 \ne H_2$, and also that $H_1 \cap H_2$ is
    nonempty of codimension~$2$. So \cref{lSymmetriesB} yields
    $H_1 \cap H_2 = H'_1 \cap H'_2$. This implies that $H'_1$ is not
    parallel to $H'_2$, and that the intersection of all four
    hyperplanes has codimension~$2$. So by the hypothesis of
    \cref{tSymmetries2D}, these four hyperplanes are in fact only two
    (distinct) ones, so $\{H_1,H_2\} = \{H'_1,H'_2\}$.
  \item[\ref{lSymmetriesA}] Assume
    $\bigl\lVert\re_{H_1}(\ve v) - \re_{H_3}(\ve v)\bigr\rVert =
    \bigl\lVert\re_{H_2}(\ve v) - \ve v\bigr\rVert$ for
    all~$\ve v \in \RR^n$. Then in particular for $\ve v \in H_2$ this
    implies $\re_{H_1}(\ve v) = \re_{H_3}(\ve v)$. As in the proof of
    \cref{lSymmetriesB}, either $\ve v \in H_1 \cap H_3$ or
    $H_1 = H_3$ follows. But the latter is impossible since it would
    imply
    $\bigl\lVert\re_{H_1}(\ve v') - \re_{H_3}(\ve v')\bigr\rVert = 0
    \ne \bigl\lVert\re_{H_2}(\ve v') - \ve v'\bigr\rVert$
    for~$\ve v' \notin H_2$, contradicting our assumption. So we
    conclude $H_2 \subseteq H_1 \cap H_3$. But an affine hypersurface
    cannot be contained in the intersection of two affine
    hypersurfaces unless they are equal, which we have already
    excluded. So our assumption leads to an unescapable
    contradiction. \qed
  \end{enumerate}
  \renewcommand{\qed}{}
\end{proof}

\section{The Cayley-Menger matrix and affine subspaces} \label{sCM}%

This section introduces some geometric tools that will be needed in
\cref{sAlgorithm}.

For vectors (or ``points'') $\ve v_0 \upto \ve v_m \in V$ in a vector
space over a field $K$, recall that the \df{affine subspace} spanned
by them, written here as $\Aff(\ve v_0 \upto \ve v_m)$, is the subset
of $V$ consisting of all linear combinations
$\sum_{i=0}^m \alpha_i \ve v_i$ with $\sum_{i=0}^m \alpha_i = 1$. Its
dimension is defined to be the dimension of the associated linear
space, which consists of all linear combinations with
$\sum_{i=0}^m \alpha_i = 0$, and which is generated by the differences
$\ve v_i - \ve v_0$.

If $V$ is a Euclidean space, the \df{Cayley-Menger matrix} (see
\mycite{Cayley:1841}) of the~$\ve v_i$ is
\begin{equation} \label{eqCMMatrix}%
  C := \begin{pmatrix}
    0 & 1 & 1 & 1 & \cdots & 1 \\
    1 & 0 & D_{0,1} & D_{0,2} & \cdots & D_{0,m} \\
    1 & D_{1,0} & 0 & D_{1,2} & \cdots & D_{1,m}  \\
    1 & D_{2,0} & D_{2,1} & 0 & \cdots & D_{2,m} \\
    \vdots & \vdots & \vdots & \vdots & \ddots & \vdots \\
    1 & D_{m,0} & D_{m,1} & D_{m,2} & \cdots & 0
  \end{pmatrix} \in \RR^{(m+2) \times (m+2)}
\end{equation}
with $D_{i,j} := \lVert\ve v_i - \ve v_j\rVert^2$. For ease of talking
about the rank of the Cayley-Menger matrix and some other matrices, we
find it convenient to introduce the \df{bordered rank} of a matrix $M$
as
\begin{equation} \label{eqBRank}%
  \brank(M) := \rank\left(
    \begin{array}{c|ccc}
      0 & 1 & \cdots & 1 \\
      \hline 1 \\
      \vdots & & M \\
      1
    \end{array}
  \right) - 2.
\end{equation}

The following results are likely to be folklore, but we could not find
a reference in the literature.

\begin{prop} \label{pCM}%
  Let $\ve v_0 \upto \ve v_m \in V$ be points in a Euclidean space and
  set $A := \Aff(\ve v_0 \upto \ve v_m)$.
  \begin{enumerate}[label=(\alph*)]
  \item \label{pCMa} With $D_{i,j} := \lVert\ve v_i - \ve v_j\rVert^2$
    we have
    \begin{multline*}
      \brank\bigl(D_{i,j}\bigr)_{i,j = 0 \upto m} =
      \brank\bigl(\langle\ve v_i,\ve v_j\rangle\bigr)_{i,j = 0 \upto
        m} = \\
      \rank\bigl(\langle\ve v_i - \ve v_0,\ve v_j - \ve
      v_0\rangle\bigr)_{i,j = 1 \upto m} = \dim(A).
    \end{multline*}
    In particular, the Cayley-Menger matrix has rank equal to $\dim(A)
    + 2$.
  \item \label{pCMb} A point $\ve w \in A$ is uniquely determined by the distances
    between~$\ve w$ and the~$\ve v_i$. More specifically, assume, after possibly
    renumbering the~$\ve v_i$, that $A = \Aff(\ve v_0 \upto \ve v_n)$ with $n = \dim(A)$.
    Then with $d_i := \lVert\ve w - \ve v_i\rVert^2$ and
    \begin{equation} \label{eqIMatrix}%
      I :=
      \begin{pmatrix}
        0 & 1 & \cdots & 1 \\
        1 & \langle\ve v_0,\ve v_0\rangle & \cdots & \langle\ve
        v_0,\ve v_n\rangle \\
        \vdots & \vdots & & \vdots \\
        1 & \langle\ve v_n,\ve v_0\rangle & \cdots &\langle\ve v_n,\ve
        v_n\rangle
      \end{pmatrix} \in \RR^{(n+2) \times (n+2)},
    \end{equation}
    the $\alpha_i$ given by
    \begin{equation} \label{eqAlpha}%
      \begin{pmatrix}
        \alpha_0 \\ \vdots \\ \alpha_n
      \end{pmatrix} :=
      \begin{pmatrix}
        0 & 1/2 & & \\
        \vdots & & \ddots \\
        0 & & & 1/2
      \end{pmatrix} I^{-1}
      \begin{pmatrix}
        2 \\
        \lVert\ve v_0\rVert^2 - d_0 \\
        \vdots \\
        \lVert\ve v_n\rVert^2 - d_n
      \end{pmatrix}        
    \end{equation}
    satisfy $\ve w = \sum_{i=0}^n \alpha_i \ve v_i$ and
    $\sum_{i=0}^n \alpha_i = 1$.
  \end{enumerate}
\end{prop}

\begin{rem} \label{rQuadratic}%
  A generalized version of \cref{pCM}\ref{pCMa} concerns the situation
  where $V$ is a vector space over a field $K$ of characteristic $\ne 2$
  equipped with a quadratic form~$q$. Then the above rank formula
  holds with $D_{i,j} := q(\ve v_i - \ve v_j)$ and
  $\langle\cdot,\cdot\rangle$ the bilinear form associated to~$q$, and
  furthermore with $\dim(A)$ replaced by $\rank(q|_{_U})$, the rank
  of~$q$ restricted to the linear space associated to $A$.
\end{rem}

\begin{proof}[Proof of \cref{pCM,rQuadratic}]
  We start with proving \cref{rQuadratic}, from which \cref{pCM}\ref{pCMa}
  follows as a special case. Form the matrix $I \in K^{(m+2) \times (m+2)}$ as 
  in~\cref{eqIMatrix}, using all the vectors~$\ve v_0 \upto \ve v_m$, not 
  just the first~$n$ of them. Now subtract the second row of $I$ from every 
  row below it, and then the second column from every column to the right of it. 
  The result is the matrix
  \[
    \tilde I = \left(\begin{array}{cc|ccc}
        0 & 1 & 0 & \cdots & 0 \\
        1 & \ast & \ast & \cdots & \ast \\
        \hline%
        0 & \ast \\
        \vdots & \vdots & & G \\
        0 & \ast
    \end{array}\right),
  \]
  where the stars stand for entries that we do not need to specify, 
  and the $(i,j)$-th entry of $G \in K^{m \times m}$ is
  \begin{equation} \label{eqViVj}
    \langle\ve v_i,\ve v_j\rangle - \langle\ve v_0,\ve v_j\rangle - 
    \langle\ve v_i,\ve v_0\rangle + \langle\ve v_0,\ve v_0\rangle =
    \langle\ve v_i - \ve v_0,\ve v_j - \ve v_0\rangle.
  \end{equation}
  We have
  \[
    \brank\bigl(\langle\ve v_i,\ve v_j\rangle\bigr)_{i,j = 0 \upto
        m} = \rank(I) - 2 = \rank(\tilde I) - 2 = \rank(G),
  \]
  so the second equation in the formula in \cref{pCM}\ref{pCMa} is proved.
  We can also start with the Cayley-Menger matrix $C$ (see~\cref{eqCMMatrix})
  with $D_{i,j} := q(\ve v_i - \ve v_j)$, and perform the same row and column
  operations that we performed on $I$. The resulting matrix is
 \[
    \tilde C = \left(\begin{array}{cc|ccc}
        0 & 1 & 0 & \cdots & 0 \\
        1 & \ast & \ast & \cdots & \ast \\
        \hline%
        0 & \ast \\
        \vdots & \vdots & & H \\
        0 & \ast
    \end{array}\right),
  \]
  where the $(i,j)$-th entry of $H \in K^{m \times m}$ is
  \begin{multline*}
    D_{i,j} - D_{0,j} - D_{i,0} = \langle\ve v_i - \ve v_j,\ve v_i -
    \ve v_j\rangle - \langle\ve v_0 - \ve v_j,\ve v_0 - \ve v_j\rangle - 
    \langle\ve v_i - \ve v_0,\ve v_i - \ve v_0\rangle = \\
    -2 \langle\ve v_i,\ve v_j\rangle + 2 \langle\ve v_0,\ve v_j\rangle
    + 2 \langle\ve v_i,\ve v_0\rangle - 2 \langle\ve v_0,\ve v_0\rangle,
  \end{multline*}
  so $H = -2 G$ by~\cref{eqViVj}. As above, we get 
  $\brank\bigl(D_{i,j}\bigr)_{i,j = 0 \upto m} = \rank(G)$, and the 
  second equation in the formula in \cref{pCM}\ref{pCMa} follows.

  It remains to show that $\rank(G) = \rank(q|_{_U})$. Replacing 
  every~$\ve v_i$ by~$\ve v_i - \ve v_0$ does not change $G$ or the
  subspace $U$, so making this replacement we may assume $\ve v_0 = 0$.
  Now $U$ is generated (as a vector space) by the~$\ve v_i$, and we may
  also replace $V$ by $U$. Since $V$ is now finite-dimensional we may 
  assume $V = K^n$. So we have $\langle\ve v,\ve w\rangle = \ve v^T A \ve w$
  with $A \in K^{n \times n}$ symmetric, and $\rank(A) = \rank(q) 
  = \rank(q|_{_U})$. So we need to show $\rank(G) = \rank(A)$.

  By~\cref{eqViVj} and since~$\ve v_0 = 0$, the entries of $G$ are the
  $\langle\ve v_i,\ve v_j\rangle$, so With $E := (\ve v_1|\ve v_2|\cdots|\ve v_m) 
  \in K^{n \times m}$ we have $G = E^T A E$. Since the~$\ve v_i$ generate $V$, 
  $E$ has rank~$n$, so the linear map $K^m \to K^n$ defined by $E$ is surjective.
  Likewise, the linear map $K^n \to K^m$ defined by $E^T$ is injective, so the map
  given by $E^T A E = G$ has an image of dimension equal to $\rank(A)$.
  This shows $\rank(G) = \rank(A)$, so the proof of \cref{rQuadratic} and
  \cref{pCM}\ref{pCMa} is finished.

  Now we prove \cref{pCM}\ref{pCMb}. The hypothesis $\ve w \in A$ implies that there exist
  $\alpha_0 \upto \alpha_n \in \RR$ such that $\ve w = \sum_{i=0}^n \alpha_i \ve v_i$ and
  $\sum_{i=0}^n \alpha_i = 1$. With $I$ as defined in~\cref{eqIMatrix}, we get
  \[
  I \cdot \begin{pmatrix}
    - \lVert\ve w\rVert^2 \\
    2 \alpha_0 \\
    \vdots \\
    2 \alpha_n
  \end{pmatrix} = \begin{pmatrix}
    2 \\
    - \lVert\ve w\rVert^2 + 2 \langle\ve v_0,\ve w\rangle \\
    \vdots \\
    - \lVert\ve w\rVert^2 + 2 \langle\ve v_n,\ve w\rangle
  \end{pmatrix} = \begin{pmatrix}
    2 \\
    \lVert\ve v_0\rVert^2 - \lVert\ve w - \ve v_0\rVert^2 \\
    \vdots \\
    \lVert\ve v_n\rVert^2 - \lVert\ve w - \ve v_n\rVert^2
  \end{pmatrix} = \begin{pmatrix}
      2 \\
    \lVert\ve v_0\rVert^2 - d_0 \\
    \vdots \\
    \lVert\ve v_n\rVert^2 - d_n
  \end{pmatrix}
  \]
  Since $I$ is invertible by \cref{pCM}\ref{pCMa}, the desired 
  equation~\cref{eqAlpha} follows from this.
\end{proof}

If we have points $\ve v_0 \upto \ve v_n$ in an $n$-dimensional
Euclidean space $V$ not lying in a proper affine subspace (i.e.,
$V = \Aff(\ve v_0 \upto \ve v_n)$), then by \cref{pCM}\ref{pCMb} any
$\ve w \in V$ can determined from the distances
$\lVert\ve w - \ve v_i\rVert$. Given several such
points~$\ve w_1 \upto \ve w_m$, their mutual distances
$\lVert\ve w_i - \ve w_j\rVert$ can then be worked out. The following
result allows the computation of the mutual distances in a direct way,
and without any knowledge of the~$\ve v_i$; only the distances between
them are needed.

\begin{prop} \label{pDistances}%
  Let $\ve v_0 \upto \ve v_n \in V$ be points in an $n$-dimensional
  Euclidean space that do not lie in a proper affine subspace, and let
  $C \in \RR^{(n+2) \times (n+2)}$
  be their Cayley-Menger matrix. Let $\ve w_1 \upto \ve w_m \in V$ be
  further points and form the matrix
  $\Delta \in \RR^{(n + 2) \times m}$ whose $i$-th column is
  $\bigl(1,\lVert\ve w_i - \ve v_0\rVert^2 \upto \lVert\ve w_i - \ve
  v_n\rVert^2)^T$. Then we have
  \begin{equation} \label{eqDistances}%
    \begin{pmatrix}
      0 & \lVert\ve w_1 - \ve w_2\rVert^2 & \lVert\ve w_1 - \ve
      w_3\rVert^2 & \cdots & \lVert\ve w_1 - \ve w_m\rVert^2 \\
      \lVert\ve w_2 - \ve w_1\rVert^2 & 0 & \lVert\ve w_2 - \ve
      w_3\rVert^2 & \cdots & \lVert\ve w_2 - \ve w_m\rVert^2 \\
      \lVert\ve w_3 - \ve w_1\rVert^2 & \lVert\ve w_3 - \ve
      w_2\rVert^2  & 0 & \cdots & \lVert\ve w_3 - \ve w_m\rVert^2 \\
      \vdots & \vdots & \vdots & \ddots & \vdots \\
      \lVert\ve w_m - \ve w_1\rVert^2 & \lVert\ve w_m - \ve
      w_2\rVert^2 & \lVert\ve w_m - \ve w_2\rVert^2 & \cdots & 0
    \end{pmatrix} = \Delta^T C^{-1} \Delta.
  \end{equation}
\end{prop}

\begin{proof}
  By \cref{pCM}\ref{pCMa}, $C$ has rank $n+2$, so $C^{-1}$ exists. Let
  us first consider the case of two points $\ve w,\ve w' \in V$. For
  ease of notation write $d_i := \lVert\ve w - \ve v_i\rVert^2$,
  $e_i := \lVert\ve w' - \ve v_i\rVert^2$, and
  $f := \lVert\ve W - \ve w'\rVert^2$.  So with
  $D_{i,j} := \lVert\ve v_i - \ve v_j\rVert^2$, the Cayley-Menger
  matrix of the points $\ve v_0 \upto \ve v_n,\ve w,\ve w'$ is
  \[
    \tilde C = \begin{pmatrix}
      0 & 1 & 1 & \cdots & 1 & 1 & 1 \\
      1 & 0 & D_{0,1} & \cdots & D_{0,n} & d_0 & e_0\\
      1 & D_{1,0} & 0 & \cdots & D_{1,n}  & d_1 & e_1\\
      \vdots & \vdots & \vdots & \ddots & \vdots & \vdots &
      \vdots \\
      1 & D_{n,0} & D_{n,1} & \cdots & 0 & d_n & e_n \\
      1 & d_0 & d_1 & \cdots & d_n & 0 & f \\
      1 & e_0 & e_1 & \cdots & e_n & f & 0
    \end{pmatrix} = \left(
      \begin{array}{ccccc|cc}
        &&&&& 1 & 1 \\
        &&&&& d_0 & e_0 \\
        && C &&& d_1 & e_1 \\
        &&&&& \vdots & \vdots \\
        &&&&& d_n & e_n \\ \hline
        1 & d_0 & d_1 & \cdots & d_n & 0 & f \\
        1 & e_0 & e_1 & \cdots & e_n & f  & 0
      \end{array} \right).
  \]
  We have
  \[
    \begin{pmatrix}
      1 & d_0 & d_1 & \cdots & d_n \\
      1 & e_0 & e_1 & \cdots & e_n
    \end{pmatrix} = \left(\begin{pmatrix}
        1 & d_0 & d_1 & \cdots & d_n \\
        1 & e_0 & e_1 & \cdots & e_n
      \end{pmatrix} C^{-1}\right) \cdot C,
  \]
  which tells us how the rows in the block below $C$ can be written in
  a unique way as a linear combination of the rows of $C$.  Since
  $\rank(\tilde C) = \rank(C)$ (again by \cref{pCM}\ref{pCMa}), the
  same linear combination of the $n + 2$ upper rows must represent the
  two bottom rows. Therefore
  \[
    \begin{pmatrix}
      0 & f \\
      f & 0
    \end{pmatrix} = \begin{pmatrix}
      1 & d_0 & d_1 & \cdots & d_n \\
        1 & e_0 & e_1 & \cdots & e_n
    \end{pmatrix} C^{-1}
    \begin{pmatrix}
      1 & 1 \\
      d_0 & e_0 \\
      d_1 & e_1 \\
      \vdots & \vdots \\
      d_n & e_n
    \end{pmatrix}.
  \]
  We obtain
  \[
    \lVert\ve w - \ve w'\rVert^2 = f = \begin{pmatrix} 1 & d_0 & d_1 &
      \cdots & d_n\end{pmatrix} \cdot C^{-1}
    \cdot \begin{pmatrix} 1 \\ e_0 \\ e_1 \\ \vdots \\
      e_n\end{pmatrix} = \begin{pmatrix} 1 \\ \lVert\ve w - \ve
      v_0\rVert^2 \\ \lVert\ve w - \ve v_1\rVert^2 \\ \vdots \\
      \lVert\ve
      w - \ve v_n\rVert^2\end{pmatrix}^T C^{-1} \begin{pmatrix} 1 \\
      \lVert\ve w' - \ve v_0\rVert^2 \\ \lVert\ve w' - \ve v_1\rVert^2
      \\ \vdots \\
      \lVert\ve w' - \ve v_n\rVert^2\end{pmatrix}.
  \]
  Since this holds for any two points $\ve w$ and $\ve w'$, the
  equation~\cref{eqDistances} follows.
\end{proof}

\section{A Reconstruction Algorithm for the Vehicle Path and Source
  Position } \label{sAlgorithm}

In this section we present an algorithm that is run each time a
loudspeaker emits a sound signal as the vehicle moves along a path.
The vehicle, on which the data is acquired and the computations are
performed, is assumed to know when the signal is emitted. For example,
the vehicle and the loudspeaker might share a common clock and follow
a predetermined signal firing schedule, or they might be connected so
that the vehicle can tell the loudspeaker when to emit a
signal. However, the vehicle does not need to know the loudspeaker
position.

When possible, the algorithm computes the position of the vehicle at
the time where the sound is emitted by the loudspeaker. It also
computes the position of the sound sources (mirror points or loudspeaker)
that were heard by the four microphones. Over time, the list of
reconstructed sources grows to include more and more sources as they
are being discovered.

When for the first time at least four noncoplanar sound sources have been detected, their positions are stored relative to a coordinate system that is defined by the current location of the vehicle. This coordinate system will be frozen and used at all times. In each subsequent call, the algorithm seeks to match at least four noncoplanar detected sound sources with sound sources that have previously been detected. This enables the algorithm to determine the vehicle position, and express the positions of the newly detected sound sources in terms of the coordinate system that has been frozen before.
If desired, the frozen coordinate system can later be recalibrated to some other coordinate system.

\newcommand{\detd}{{\operatorname{detected}}}%
\newcommand{\know}{{\operatorname{known}}}

\begin{longalg}[Self-location and detecting environment geometry] \label{aLocate}%
  \mbox{}%
  \begin{description}
  \item[Input (optional)] A list of points
    $\ve s_1 \upto \ve s_n \in \RR^3$, which are known positions of 
    sound sources. 
    The $\ve s_i$ may not be coplanar, so in
    particular $n \ge 4$. The list of~$\ve s_i$ is either taken from
    previous runs of the algorithm or passed to it after preparing the
    room in the room-coordinates scenario (see above).

    In both scenarios it is assumed that the algorithm knows the
    coordinate vectors $\ve m_1 \upto \ve m_4 \in \RR^3$ of the
    positions of the microphones with respect to the coordinate system
    given by the principal axes (roll, pitch and yaw) of the vehicle.
  \item[Output] In case of success, the output consists of:
    \begin{enumerate}[label=(\alph*)]
    \item \label{Outa} an updated list of known positions of sound
      sources~$\ve s_i$, which can be used as input for the next call, and,
    \item \label{Outb} if input was provided: the present location of the vehicle,
      given by the position~$\ve v \in \RR^3$ of its center of mass
      and by a matrix in $A \in \Or_3(\RR)$ whose columns give the
      present directions of the principal axes. (If no input was provided,
      the output consists only of what is described in~\ref{Outa}.)
    \end{enumerate}
    If unsuccessful, the algorithm returns ``FAIL.'' In this case, the
    list of known sound sources from the last successful call
    remains unchanged and should be used for the next call.
  \end{description}
  \begin{enumerate}[label=(\arabic*)]
  \item \label{aLocate1} {\bf Data collection:} After a sound has been
    emitted by the loudspeaker, for each
    $i = 1 \upto 4$ record the signals from this sound and its
    first-order echoes as received by the $i$-th microphone. From the
    times of reception, calculate the distances travelled by the
    signals from emission to reception, and for each microphone
    collect the squares of these distances in a set $\mathcal D_i$.
  \item \label{aLocate2} {\bf Echo matching:} With
    $f_D(x_1,x_2,x_3,x_4)$ given by Equation~\cref{eqCMPolynomial}
    below, form the matrix $\Delta \in \RR^{4 \times m}$ whose columns
    are the $(d_1,d_2,d_3,d_4)^T$ such that
    $f_D(d_1,d_2,d_3,d_4) = 0$, where $(d_1 \upto d_4)$ ranges through
    the cartesian product
    $\mathcal D_1 \times \cdots \times \mathcal D_4$. Then the columns
    of $\Delta$ correspond to the detected sound sources, and
    in each column the $i$-th entry is the squared distance between
    the $i$-th microphone and the sound source.
  \item \label{aLocate3} {\bf Compute the distance matrix:} With
    $\Delta_{i,j}$ the entries of $\Delta$, form the matrix
    \[
      \overline\Delta := \left(
        \begin{array}{ccc}
          1 & \cdots & 1 \\
          \Delta_{1,1} & \cdots & \Delta_{1,m} \\
          \vdots &  & \vdots \\
          \Delta_{4,1} & \cdots & \Delta_{4,m} 
        \end{array}\right) \in \RR^{5 \times m}.
    \]
    With $C \in \RR^{5 \times 5}$ the Cayley-Menger matrix of
    the~$\ve m_i$ given in Equation~\cref{eqCMMatrix}, compute the
    matrix%
    \[
      D^\detd := \overline\Delta^T C^{-1} \overline\Delta \in \RR^{m \times
        m}.
    \]
    Then $D^\detd$ stores the squared distances between the detected
    sound sources.
  \item \label{aLocate4} {\bf Case distinctions:} If
    $\brank(D^\detd) < 3$ (see~\cref{eqBRank} for the definition of the bordered rank), 
    return ``FAIL'' and skip the remaining steps. The condition on the bordered rank
    means that the detected sound sources are coplanar.

    \noindent If no points~$\ve s_i$ have been passed as input to the
    algorithm, set $\ve b_i := \ve m_i$ ($i = 1 \upto 4$) and go to
    step~\ref{aLocate7}.
  \item \label{aLocate5} {\bf Submatrix Matching:} Form the matrix
    $D^\know \in \RR^{n \times n}$ with entries
    $\lVert\ve s_i - \ve s_j\rVert^2$.

    \noindent If possible, find indices
    $i_1 \upto i_4 \in \{1 \upto m\}$ and
    $j_1 \upto j_4 \in \{1 \upto n\}$ such that, with notation
    explained below,
    \[
      D^\detd_{i_1 \upto i_4} = D^\know_{j_1 \upto j_4}
    \]
    and such that $\brank(D^\detd_{i_1 \upto i_4}) = 3$. See
    \cref{aMatch} below for an efficient way to find the $i$'s
    and~$j$'s.
    
    \noindent If no such $i_1 \upto i_4$ and $j_1 \upto j_4$ exist,
    return ``FAIL'' and skip the remaining steps of the algorithm.

    \noindent If they do exist, the mutual distances between the known
    points $\ve s_{j_1} \upto \ve s_{j_4}$ are the same as the mutual
    distances between the detected points corresponding to the columns
    of $\Delta$ with numbers $i_1 \upto i_4$. The remaining steps work
    with the assumption that these detected points {\em are}
    $\ve s_{j_1} \upto \ve s_{j_4}$. Set $\ve b_i := \ve s_{j_i}$
    ($i = 1 \upto 4$).
  \item \label{aLocate6} {\bf Self-locating:} Set
    \[
      M := \left(
        \begin{array}{cccc}
          \ve m_1 & \ve m_2 & \ve m_3 & \ve m_4 \\
          \hline 1 & 1 & 1 & 1
        \end{array}
      \right) \in \RR^{4 \times 4}.
    \]
    and let $B \in \RR^{3 \times 4}$ be the upper $3 \times 4$-part of
    the transpose-inverse
    \[
      \left(
        \begin{array}{cccc}
          \ve b_1 & \ve b_2 & \ve b_3 & \ve b_4 \\
          \hline 1 & 1 & 1 & 1
        \end{array}
      \right)^{-T} \in \RR^{4 \times 4}.
    \]
    Compute
    \[
      \bigl(A\mid\ve v\bigr) := \frac{1}{2} B \cdot \left(\lVert\ve
        b_j\rVert^2 - \Delta_{k,i_j}\right)_{j,k = 1 \upto 4} \cdot
      M^{-1}
    \]
    with $A \in \RR^{3 \times 3}$ and $\ve v \in \RR^3$. Then, as we
    will see in the proof of \cref{tLocate}, $\ve v$ is the present
    position of the vehicle's center of mass, $A \in \Or_3(\RR)$, and
    the columns of $A$ give the present directions of its principal
    axes.

    \noindent To prepare for step~\ref{aLocate7}, set
    $\Delta_{j,k} := D^\detd_{i_j,k}$, the $(i_j,k)$-entry of
    $D^\detd$ ($j = 1 \upto 4$, $k = 1 \upto m$). This is the squared
    distance between~$\ve b_j$ and the $k$-th detected point.
    
    \noindent Return~$\ve v$ and $A$ (as the output described in~\ref{Outb}),
    and continue with step~\ref{aLocate7}.
  \item \label{aLocate7} {\bf Knowledge update:} With
    $B$ defined as in step~\ref{aLocate6},
    compute the points
    \[
      \ve t_k := \frac{1}{2} B \cdot
      \begin{pmatrix}
        \lVert\ve b_1\rVert^2 - \Delta_{1,k} \\ 
        \vdots \\
        \lVert\ve b_4\rVert^2 - \Delta_{4,k} 
      \end{pmatrix}
    \]
    ($k = 1 \upto m$). As shown in the proof of \cref{tLocate}, these
    are the detected sound sources. Update the list
    $\ve s_1 \upto \ve s_n$ by adding those $\ve t_k$ that are not
    equal to one of the $\ve s_i$.

    \noindent Return the updated list of~$\ve s_i$ as the output described in~\ref{Outa}.
  \end{enumerate}
\end{longalg}

In the following we explain some notation used in the algorithm and
make some remarks.

\begin{description}
\item[Step~\ref{aLocate2}] From the coordinate vectors $\ve m_i$ of
  the microphone positions, the algorithm can compute their squared
  mutual distances $D_{i,j} = \lVert\ve m_i - \ve m_j\rVert^2$. From
  these, it can form the Cayley-Menger matrix $C \in \RR^{5 \times 5}$
  according to~\cref{eqCMMatrix} and the \df{Cayley-Menger
    polynomial}
  \begin{equation} \label{eqCMPolynomial}%
    f_D(x_1,x_2,x_3,x_4) = \det
    \left(
      \begin{array}{ccccc|cc}
        &&&&& 1 \\
        &&&&& x_1 \\
        && C &&& x_2 & \\
        &&&&& x_3 \\
        &&&&& x_4 \\ \hline
        1 & x_1 & x_2 & x_3 & x_4 & 0
      \end{array}
    \right).
  \end{equation}
\item[Step~\ref{aLocate5}] If $i_1 \upto i_r \in \{1 \upto m\}$, we write
  $D_{i_1 \upto i_r} \in \RR^{r \times r}$ is the submatrix obtained
  by selecting the rows and columns with indices $i_1 \upto i_r$, or
  more formally
  $D_{i_1 \upto i_r} := \bigl(d_{i_j,i_k}\bigr)_{j,k = 1 \upto r}$.
\item[Step~\ref{aLocate6}] From the matrix
  $A = \left(\begin{smallmatrix} a_{1,1} & a_{1,2} & a_{1,3} \\
      a_{2,1} & a_{2,2} & a_{2,3} \\ a_{3,1} & a_{3,2} &
      a_{3,3} \end{smallmatrix}\right)$, the present yaw, pitch and
  roll angles~$\alpha$, $\beta$ and~$\gamma$, respectively, can be
  easily determined by the well-known formulas%
  \newcommand{\atan}{\operatorname{atan2}}
  \[
    \alpha = \atan(a_{2,1},a_{1,1}), \quad \beta = -
    \arcsin(a_{3,1}) \quad \text{and} \quad \gamma =
    \atan(a_{3,2},a_{3,3})
  \]
  (if $|a_{3,1}| \ne 1$; otherwise $\beta = - a_{3,1} \cdot \pi/2$,
  $\gamma = 0$ and $\alpha = \atan(-a_{1,2},a_{2,2})$ is a non-unique
  solution).  There is a small caveat: for these formulas to give the
  expected values, the coordinate system used by the algorithm needs
  to be oriented as the vehicle-coordinate system (right- or
  left-handed), and its $z$-axis needs to point up or down in
  accordance with the vehicle's yaw axis.

 \item[Step~\ref{aLocate7}] Some heuristics may be applied to only add
  such $\ve t_k$ into the list that are ``sufficiently far away'' from
  points already in the list. Moreover, if the list grows so long that
  it renders step~\ref{aLocate5} inefficient, points may be deleted
  from the list, as long as the points still in the list do not become
  coplanar. But since \cref{aMatch} is generically only quadratic
  in~$n$, this should be unlikely to happen.
\end{description}

\begin{theorem} \label{tLocate}%
  \cref{aLocate} is correct under the following assumptions:
  \begin{enumerate}[label=(\alph*)]
  \item \label{tLocateA} No ghost walls are detected.
  \item \label{tLocateB} For sound sources
    $\ve s_1 \upto \ve s_4,\ve s'_1 \upto \ve s'_4$ such that the
    $\ve s_i$ are not coplanar, the condition that
    $\lVert\ve s_i - \ve s_j\rVert = \lVert\ve s'_i - \ve s'_j\rVert$
    for $1 \le i < j \le 4$ implies that $\ve s_i = \ve s'_i$ for
    all~$i$.
  \end{enumerate}
\end{theorem}

So due to~[\citenumber{Boutin:Kemper:2019},\citenumber{Boutin:Kemper:2022}],
the hypothesis~\ref{tLocateA} is satisfied for almost all vehicle positions, 
and due to \cref{tSymmetries3D}, the hypothesis~\ref{tLocateB} is satisfied for
almost all loudspeaker positions.

\begin{proof}
  The correctness of step~\ref{aLocate2} is the very definition of
  ``no ghost walls.'' Step~\ref{aLocate3} is correct because of
  \cref{pDistances}, but there is one subtlety to observe:
  $\Delta_{i,j}$ is the squared distance between the $j$-th detected
  source and the (unknown) position $\tilde{\ve m}_i$ of the
  $i$-th microphone at the time when the algorithm was called, so
  \cref{pDistances} has to be used with the $\tilde{\ve m}_i$ instead
  of the time-independent (and known) coordinate vectors~$\ve
  m_i$. But since the $\ve m_i$ are coordinate vectors with respect to
  a cartesian coordinate system, we always have
  $\lVert\tilde{\ve m}_i - \tilde{\ve m}_j\rVert = \lVert\ve m_i - \ve
  m_j\rVert$, so setting up the matrix $C$ with the $\ve m_i$ instead
  of the $\tilde{\ve m}_i$ does yield the correct result.

  In step~\ref{aLocate4}, the first condition guarantees that the
  algorithm only proceeds if the detected sound sources are
  not coplanar (this follows from \cref{pCM}\ref{pCMa}), and in particular
  there are at least four of them. The
  second ``If''-statement applies if the algorithm was called without
  input, which can only happen in the first call in which non-coplanar
  sources were detected. The correctness of the assumption made
  in step~\ref{aLocate5} is guaranteed by assumption~\ref{tLocateB} of
  the theorem.

  The main part of the proof concerns steps~\ref{aLocate6}
  and~\ref{aLocate7}. Both use ``reference points''
  $\ve b_1 \upto \ve b_4$. For any vector $\ve w \in \RR^3$ we have
  \[
    \left(
      \begin{array}{cccc}
        \ve b_1 & \ve b_2 & \ve b_3 & \ve b_4 \\
        \hline 1 & 1 & 1 & 1
      \end{array}
    \right)^T \left(
      \begin{array}{c}
        \\ 2 \ve w \\ \\ \hline -\lVert\ve w\rVert^2
      \end{array}
    \right) =
    \begin{pmatrix}
      2 \langle\ve b_1,\ve w\rangle - \lVert\ve w\rVert^2 \\
      \vdots \\
      2 \langle\ve b_4,\ve w\rangle - \lVert\ve w\rVert^2
    \end{pmatrix} =
    \begin{pmatrix}
      \lVert\ve b_1\rVert^2 - \lVert\ve b_1 - \ve w\rVert^2 \\
      \vdots \\
      \lVert\ve b_4\rVert^2 - \lVert\ve b_4 - \ve w\rVert^2
    \end{pmatrix}.
  \]
  The~$\ve b_j$ used in the algorithm are not coplanar, which means
  that the matrix on the left is invertible. So if
  $B \in \RR^{3 \times 4}$ is as in steps~\ref{aLocate6}
  and~\ref{aLocate7} and if
  $\Delta_l := \lVert\ve b_l - \ve w\rVert^2$, then
  \begin{equation} \label{eqNail}%
    \ve w = \frac{1}{2} B \cdot
    \begin{pmatrix}
      \lVert\ve b_1\rVert^2 - \Delta_1 \\
      \vdots \\
      \lVert\ve b_4\rVert^2 - \Delta_4
    \end{pmatrix}.
  \end{equation}
  Let us write $\ve t_1 \upto \ve t_m$ for the positions of the
  detected sound sources (which step~\ref{aLocate7} seeks to
  work out). So $\ve t_{i_l} = \ve s_{j_l}$ for $l = 1 \upto 4$ by the
  assumption made in step~\ref{aLocate5}.

  Now step~\ref{aLocate7} can be reached directly from
  step~\ref{aLocate4}, or from step~\ref{aLocate6}. In the first case
  we have $\ve b_l = \ve m_l$ and
  $\Delta_{l,k} = \lVert\ve m_l - \ve t_k\rVert^2 = \lVert\ve b_l -
  \ve t_k\rVert^2$ (from step~\ref{aLocate2}), and in the second case
  $\ve b_l = \ve s_{j_l} = \ve t_{i_l}$ (from step~\ref{aLocate5}) and
  $\Delta_{l,k} = D^\detd_{i_l,k} = \lVert\ve t_{i_l} - \ve
  t_k\rVert^2 = \lVert\ve b_l - \ve t_k\rVert^2$. So in both cases the
  formula for~$\ve t_k$ in step~\ref{aLocate7} is correct
  by~\cref{eqNail}. Notice that in the first case (which happens in
  the case of zero input, and only in the first successful call of the
  algorithm), the points $\ve b_l = \ve m_l$ are represented according
  to the coordinate system given by the principal axes of the vehicle
  {\em at the time when the algorithm was called}. Therefore
  the~$\ve t_k$ are also represented according to this coordinate
  system. These are fed back into the algorithm in subsequent calls
  and serve, together with sound sources detected later, as
  reference points. It follows that all detected sound sources
  will be given according to this coordinate system. Being virtual
  sound sources, they remain fixed even as the vehicle moves on, so
  the coordinate system is also fixed once and for all. If, on the
  other hand, some input is given to the initial call of the
  algorithm, then the coordinate system according to which the input
  is given remains unchanged throughout.

  Step~\ref{aLocate6} is always called with reference points
  $\ve b_l = \ve s_{j_l} = t_{i_l}$, which is given according to the
  permanently chosen coordinate system. The $\Delta_{k,i_l}$ come from
  step~\ref{aLocate2}, so they are the squared distance between
  $\ve t_{i_l} = \ve b_l$ and the~$k$-th microphone {\em at the time
    when the algorithm was called}, which we write as
  $\tilde{\ve m}_k$ as before. So
  $\Delta_{k,i_l} = \lVert\ve b_l - \tilde{\ve m}_k\rVert^2$,
  and~\cref{eqNail} shows that
  \begin{equation} \label{eqTilde}%
    \frac{1}{2} B \cdot \left(\lVert\ve
      b_l\rVert^2 - \Delta_{k,i_l}\right)_{l,k = 1 \upto 4} =
    \bigl(\tilde{\ve m}_1 \mid \cdots \mid \tilde{\ve m}_4\bigr).
  \end{equation}
  Now in contrast to the~$\tilde{\ve m}_k$, the~$\ve m_k$ are the
  coordinate vectors of the microphone positions with respect to the
  principal axes of the vehicle. Since the microphones are mounted on
  the vehicle, these coordinate vectors remain constant, so in
  particular they apply to the present position of the
  microphones. Since the origin of the vehicle-coordinate system
  is~$\ve v$, the center of mass, this means that
  $\tilde{\ve m}_k - \ve v$ is a linear combination of the unit
  vectors ~$\ve x$, $\ve y$ and~$\ve z$ defining the present
  directions of the principal axis, with the coefficients of the
  linear combination given by the components of~$\ve m_k$. With
  $A := \bigl(\ve x \mid \ve y \mid \ve z\bigr)$, we can write this as
  $\tilde{\ve m}_k = A \cdot \ve m_k + \ve v$, or in matrix form
  $\bigl(\tilde{\ve m}_1 \mid \cdots \mid \tilde{\ve m}_4\bigr) =
  \bigl(A \mid \ve v\bigr) \cdot M$ with $M$ as defined in
  step~\ref{aLocate6}. Combining this with~\cref{eqTilde} shows that
  the formula for $\bigl(A \mid \ve v\bigr)$ in step~\ref{aLocate6} is
  correct. Since~$\ve x$, $\ve y$ and $\ve z$ are perpendicular unit
  vectors, $A \in \Or_3(\RR)$ follows.
\end{proof}


The following algorithm is a ``subroutine'' of \cref{aLocate}. We
formulate it in a slightly more general form.

\begin{longalg}[Find matching submatrices] \label{aMatch}%
  \mbox{}%
  \begin{description}
  \item[Input] Two symmetric matrices
    $A = (a_{i,j}) \in \RR^{m \times m}$ and
    $B = (b_{i,j}) \in \RR^{n \times n}$, and an integer~$r$ with
    $1 \le r \le \min\{m,n\}$.
  \item[Output] Integers $i_1 \upto i_r \in \{1 \upto m\}$ and
    $j_1 \upto j_r \in \{1 \upto n\}$ with the $i_\nu < i_\mu$ and the
    $j_\nu \ne j_\mu$ for $\nu < \mu$, such that
    $A_{i_1 \upto i_r} = B_{j_1 \upto j_r}$ and
    $\brank\bigl(A_{i_1 \upto i_r}\bigr) = r-1$, or ``FAIL'' if no
    such $i_\nu$ and $j_\nu$ exist.
  \end{description}
  \begin{enumerate}[label=(\arabic*)]
  \item \label{aMatch1} Set $k := 1$ and $i_1 := j_1 := 1$.
  \item \label{aMatch2} WHILE $k \le r$ DO
    \begin{enumerate}[label=(\arabic*)]
      \setcounter{enumii}{\arabic{enumi}}
    \item \label{aMatch3} IF $j_k \notin \{n+1,j_1 \upto j_{k-1}\}$,
      $a_{i_k,i_\nu} = b_{j_k,j_\nu}$ for $1 \le \nu \le k$, and
      $\brank\bigl(A_{i_1 \upto i_k}\bigr) = k-1$, THEN set
      $i_{k+1} := i_k + 1$, $j_{k+1} := 1$ and $k := k + 1$.
    \item \label{aMatch4} ELSE IF $j_k < n$, THEN set $j_k := j_k + 1$.
    \item \label{aMatch5} ELSE IF $i_k < m - r + k$, THEN set
      $i_k := i_k + 1$ and $j_k :=1$.
    \item \label{aMatch6} ELSE IF $k > 1$, THEN set
      $j_{k-1} := j_{k-1} + 1$ and $k := k-1$.
    \item \label{aMatch7} ELSE Return ``FAIL''.
    \item \label{aMatch8} END IF
    \end{enumerate}
  \item \label{aMatch9} END WHILE
  \item \label{aMatch10} Return $i_1 \upto i_r$ and $j_1 \upto j_r$.
  \end{enumerate}
\end{longalg}

The condition on the bordered rank is not an essential part of the algorithm.
For other possible applications of the algorithm, this condition can be replaced
by any other condition of interest, or omitted altogether.

\begin{theorem} \label{tMatch}%
  \cref{aMatch} terminates after finitely many steps and is
  correct. Moreover, if the $b_{i,j}$ for $1 \le i < j \le n$ are
  pairwise distinct, then the algorithm requires
  $\mathcal O(m^r n + m^2 n^2)$ operations of real numbers (almost all
  of them comparisons) for fixed~$r$.
\end{theorem}

\begin{proof}
  We compare tuples $(i_1,j_1,i_2,j_2 \upto i_k,j_k)$ of different
  lengths lexicographically with the additional rule that appending
  entries to a tuple makes it bigger, as in a real lexicon. So in all
  steps~\ref{aMatch3}--\ref{aMatch6}, the tuple is replaced by a
  strictly bigger one (which is made longer in step~\ref{aMatch3} and
  shorter in step~\ref{aMatch6}). Since the total length is bounded
  by~$r$ and the entries are bounded by $\max\{m,n+1\}$, this
  guarantees termination.

  To prove correctness, we claim that throughout the algorithm we have
  $a_{i_\nu,i_\mu} = b_{j_\nu,j_\mu}$ for $1 \le \nu,\mu < k$.  In
  fact, this is true for $k = 1$, and step~\ref{aMatch3} affords a
  proof by induction on~$k$. So if the algorithm terminates with
  returning integers~$i_\nu$ and~$j_\nu$, then $k = r + 1$ was
  reached, so indeed $A_{i_1 \upto i_r} = B_{j_1 \upto
    j_r}$. Moreover, the tuple $(i_1 \upto i_k)$ is increasing
  throughout the algorithm, and step~\ref{aMatch3} makes sure that the
  $j_\nu$ are pairwise distinct. The rank condition in that step
  provides the condition on the $\ve s_{i_\nu}$.

  Conversely, assume there exist integers
  $i_1' \upto i_r',j_1' \upto j_r'$ meeting the specifications of the
  algorithm. We claim that
  throughout the algorithm
  \begin{equation} \label{eqCompare}%
    (i_1',j_1' \upto i_k',j_k') \ge (i_1,j_1 \upto i_k,j_k)
  \end{equation}
  (comparing lexicographically). This is true after step~\ref{aMatch1}
  since $(i_1',j_1') \ge (1,1)$. Moreover, if the conditions in
  step~\ref{aMatch3} are satisfied, then also
  $(i_1',j_1' \upto i_{k+1}',j_{k+1}') \ge (i_1,j_1 \upto
  i_k,j_k,i_k+1,1)$, so~\cref{eqCompare} continues to hold. On the
  other hand, if at least one of the conditions in step~\ref{aMatch3}
  does not hold, this implies
  $(i_1',j_1' \upto i_k',j_k') \ne (i_1,j_1 \upto i_k,j_k)$, so we
  have~''$>$'' in~\cref{eqCompare}. Therefore
  $(i_1',j_1' \upto i_k',j_k') \ge (i_1,j_1 \upto i_k,j_k + 1)$. By
  the specifications of the algorithm, we have $j_k' \le n$ and
  $i_k' \le m - r + k$. So if $j_k \ge n$, then
  $(i_1',j_1' \upto i_k',j_k') \ge (i_1,j_1 \upto i_k + 1,1)$, and if
  in addition $i_k \ge m - r + k$, then even
  $(i_1',j_1' \upto i_{k-1}',j_{k-1}') \ge (i_1,j_1 \upto
  i_{k-1},j_{k-1} + 1)$. This shows that if the condition of any of
  the steps~\ref{aMatch4}--\ref{aMatch6} is satisfied,
  then~\cref{eqCompare} continues to hold, so indeed it holds
  throughout. The argument also shows that if step~\ref{aMatch6} were
  reached with $k = 1$, then $(i_1',j_1') > (i_1,j_1)$ but
  $i_1 \ge m - r + 1 \ge i_1'$, and $j_1 \ge n \ge j_1'$, a
  contradiction. So if the specifications of the algorithm are
  satisfiable, the algorithm will not return ``FAIL.''

  Finally, let us consider the running time. For each $k = 1 \upto r$
  we give an upper bound for the number of comparisons (i.e., checks
  whether $a_{i_k,i_\nu} = b_{j_k,j_\nu}$ in step~\ref{aMatch3}) that
  occur during the entire run of the algorithm for this
  particular~$k$. For $k = 1$, there are at most $m \cdot n$
  comparisons, as $i_1$ and $j_1$ range. For $k = 2$, the upper bound
  is $\binom{m}{2} \cdot n (n-1) \cdot k \le m^2 n^2$. Now $k \ge 3$
  is only reached if $b_{j_1,j_2} = a_{i_1,i_2}$, and by the
  hypothesis on the distinctness of the entries of $B$, this condition
  determines the set $\{j_1,j_2\}$, so there are only two
  possibilities for the ordered pair $(j_1,j_2)$. Thus for $k = 3$
  only $j_3$ ranges freely, and we get
  $2 \cdot (n-2) \cdot \binom{m}{3} \cdot k \le m^3 n$ as
  an upper bound. Finally, for $k > 3$ we have
  $b_{j_1,j_\nu} = a_{i_1,i_\mu}$ for $\nu < k$, in particular for
  $\nu = 2,3$. This determines the sets $\{j_1,j_2\}$ and
  $\{j_1,j_3\}$ uniquely. Because of the distinctness of the $j_\nu$,
  $j_1$ is uniquely determined as the sole element in the
  intersection, and this means that every $j_\nu$ with $\nu < k$ is
  also uniquely determined. Therefore we obtain
  $(n - k + 1) \cdot \binom{m}{k} \cdot k \le (n-k+1) \cdot m^r$ as an
  upper bound. Summing over~$k$ shows that the total number of
  comparisons is in $\mathcal O(m^r n + m^2 n^2)$.

  Further operations of real numbers are required for the rank
  determination in step~\ref{aMatch3}. For a given~$k$, the matrix
  depends only on $i_1 \upto i_k$, so there are at most $\binom{m}{k}$
  rank determinations, each requiring at most $\mathcal O(k^3)$
  operations. Since $\binom{m}{k} k^3 \le m^r r^3$, the cost for all
  $k = 1 \upto r$ lies in $\mathcal O(m^r r^4)$, which, since~$r$ is
  considered as a constant, is subsumed in $\mathcal O(m^r n)$.
\end{proof}

\begin{rem} \label{rMatch}%
  A simpler method for the same purpose as \cref{aMatch} would be to
  just try all tuples of integers $(i_1 \upto i_r,j_1 \upto j_r)$ in
  the admissible range. This requires $\mathcal O(m^r n^r)$ operations
  in $\RR$. In our application we have $r = 4$, so as a function of
  $l = \max\{m,n\}$, our algorithm has running time $\mathcal O(l^5)$,
  compared to the simpler one with $\mathcal O(l^8)$. Moreover, the
  asymmetry in~$m$ and~$n$ of the running time is fortunate since we
  apply the algorithm in a situation where~$m$ is the number of
  detected sound sources and thus has no intrinsic growth as
  the vehicle travels, and where~$n$ is the number of known virtual
  sound sources, which can be expected to grow ever larger.
\end{rem}

\section*{Acknowledgements}
The authors thank the Simons Laufer Mathematical Sciences Institute (formerly the Mathematical Sciences Research Institute (MSRI)),  the Banff International Research Station for Mathematical Innovation and Discovery (BIRS) and the Centro Internazionale per la Ricerca Matematica (CIRM) in Trento for supporting their work. 

\bibliographystyle{mybibstyle} \bibliography{bib}

\end{document}